\documentclass[11pt, a4paper]{article}

\usepackage[whole]{bxcjkjatype}

\usepackage{bm, amsmath, amsthm, amssymb, accents, comment}
\usepackage{ascmac,authblk}
\RequirePackage{amsthm,amsmath,amsfonts,amssymb}
\RequirePackage{natbib,url}
\RequirePackage{graphicx}
\usepackage{tikz}
\usepackage{pgfplots}
\usepackage{subcaption}
\pgfplotsset{
	compat=newest, 
	cycle list name=exotic }
\usepackage{algorithm}
\usepackage{algpseudocode, setspace}

\newtheorem{theorem}{Theorem}
\newtheorem{proposition}{Proposition}
\newtheorem{lemma}{Lemma}
\newtheorem{corollary}{Corollary}
\theoremstyle{remark}
\newtheorem{remark}{Remark}

\usepackage{typearea}
\typearea{12}

\begin{document}

\title{Matrix norm shrinkage estimators and priors}

\author{%
	Xiao Li$^{1}$,
	Takeru Matsuda$^{1,2}$,
	Fumiyasu Komaki$^{1,2}$\\
	\small $^{1}$Graduate School of Information Science and Technology,
	The University of Tokyo,
	7-3-1 Hongo, Bunkyo-ku, Tokyo 113-8656, Japan\\
	\small $^{2}$RIKEN Center for Brain Science,
	2-1 Hirosawa, Wako, Saitama 351-0198, Japan
}

\date{}

\maketitle

\begin{abstract}%
We develop a class of minimax estimators for a normal mean matrix under the Frobenius loss, which generalizes the James--Stein and Efron--Morris estimators.
It shrinks the Schatten norm towards zero and works well for low-rank matrices.
We also propose a class of superharmonic  priors based on the Schatten norm, which generalizes Stein's prior and the singular value shrinkage prior. 
The generalized Bayes estimators and Bayesian predictive densities with respect to these priors are minimax.
We examine the performance of the proposed estimators and priors in simulation.
\end{abstract}

\section{Introduction}
Suppose that we have a matrix observation $X \in \mathbb{R}^{n \times m}$ whose entries are independent normal random variables $X_{ij} \sim {\rm N} (M_{ij},1)$, where $n \geq m$ and $M \in \mathbb{R}^{n \times m}$ is an unknown mean matrix. 	
In matrix normal notation, it is expressed as $X \sim {\rm N}_{n,m} (M,I_n,I_m)$.
In this setting, we consider estimation of $M$ under the Frobenius loss
\[
    l(M,\hat{M}) = \| \hat{M}-M \|_{\mathrm{F}}^2 = \sum_{i=1}^n \sum_{j=1}^m (\hat{M}_{ij}-M_{ij})^2.
\]
Although the maximum likelihood estimator $\hat{M}=X$ is minimax, it is inadmissible when $nm \geq 3$ from Stein's paradox \citep{stein1974estimation}.

There are two important estimators that dominate the maximum likelihood estimator, and are thus minimax. 
One is the James--Stein (JS) estimator \citep{james1961estimation} given by
\begin{align}
\hat M_{\mathrm{JS}}=\Big(1-\frac{nm-2}{\| X \|_{\mathrm{F}}^2}\Big)X \label{JS}
\end{align}
when $nm \geq 3$.
This estimator shrinks towards the zero matrix.
The other is the Efron--Morris (EM) estimator \citep{efron1972empirical} given by
\begin{align}
\hat M_{\mathrm{EM}}=X \left( I_m-(n-m-1)(X^\top X)^{-1} \right) \label{EM}
\end{align}
when $n \geq m+2$.
This estimator shrinks towards the space of low-rank matrices.
Namely, let $X = U S V^{\top}$, $U \in \mathbb{R}^{n \times m}$, $V \in \mathbb{R}^{m \times m}$, $S = {\rm diag} (s_1, \dots, s_m)$ be the singular value decomposition of $X$, where $U^{\top} U = V^{\top} V = I_m$ and
$ s_1 \geq \cdots \geq  s_m
$ are the singular values of $X$.
Then, $\hat{M}_{{\rm EM}} = U \hat{\Sigma} V^{\top}$ with $\hat{\Sigma} = {\rm diag} (\hat{\sigma}_1, \dots, \hat{\sigma}_m)$, where
\begin{align*}
	\hat{\sigma}_i = \left( 1 - \frac{n-m-1}{s_i^2} \right) s_i, \quad i=1, \ldots, m.
\end{align*}
Whereas the James--Stein estimator shrinks towards the zero matrix as seen from
\[
\hat{M}_{\mathrm{JS}} = \frac{nm-2}{\| X \|_{\mathrm{F}}^2} O + \left(1-\frac{nm-2}{\| X \|_{\mathrm{F}}^2} \right) X,
\]
the Efron--Morris estimator can be rewritten as
\begin{align*}
    \hat{M}_{{\rm EM}} = \frac{n-m-1}{s_1^2} X_{(0)} + \sum_{k=1}^{m-1} \left( \frac{n-m-1}{s_{k+1}^2} - \frac{n-m-1}{s_k^2} \right) X_{(k)} + \left( 1 - \frac{n-m-1}{s_m^2} \right) X_{(m)},
\end{align*}
where
\begin{align*}
    X_{(k)} = U \mathrm{diag} (s_1,\dots,s_k,0,\dots,0) V^{\top}
\end{align*}
is the rank-$k$ truncation of $X$.
Note that $X_{(0)}=O$ and $X_{(m)}=X$.
In this sense, the Efron--Morris estimator shrinks $X$ towards the low rank matrices $X_{(k)}$ with strength determined by its singular values.
Thus, the Efron--Morris estimator shrinks the singular values towards zero and works well when $M$ is close to low-rank (see \cite{matsuda2022estimation} for details).
Note that the number of nonzero singular values of a matrix is equal to its rank.
Thus, the James--Stein and Efron--Morris estimators can be viewed as scalar and matricial shrinkage estimators, respectively.
Recently, \cite{yuasa2023weighted} proposed a minimax estimator that combines these two estimators by using a weight determined by minimizing an unbiased estimate of risk. 

Correspondingly, there are two important superharmonic priors of $M$, where a prior $\pi(M)$ is said to be superharmonic if
\[
    \Delta \pi (M) = \sum_{i=1}^n \sum_{j=1}^m \frac{\partial^2 \pi}{\partial M_{ij}^2} (M) \leq 0
\]
for every $M$.
One is Stein's prior \citep{stein1974estimation} given by
\begin{align}
\pi_{\mathrm{S}}(M)=\| M \|_{\mathrm{F}}^{2-nm}, \label{stein}
\end{align}
for $nm \geq 3$.
The generalized Bayes estimator with respect to Stein's prior shrinks towards the origin like the James--Stein estimator \eqref{JS}.
The other is the singular value shrinkage (SVS) prior \citep{matsuda2015singular} given by
\begin{align}
\pi_{\mathrm{SVS}}(M)=\det(M^\top M)^{-(n-m-1)/2}, \label{SVS}
\end{align}
for $n-m \geq 2$.
The generalized Bayes estimator with respect to the singular value shrinkage prior shrinks the singular values towards zero like the Efron--Morris estimator \eqref{EM}.
Since a generalized Bayes estimator with respect to a superharmonic prior is minimax \citep{stein1974estimation}, the generalized Bayes estimators with respect to these priors are minimax under the Frobenius loss.
See \cite{tsukuma2008admissibility,tsukuma2017proper} for other types of minimax (generalized) Bayes estimators.

Superharmonic priors work well for Bayesian prediction as well \citep{komaki2001shrinkage,george2006improved,matsuda2015singular}.
Specifically, suppose that we observe $X \sim {\rm N}_{n,m} (M, I_n, I_m)$ and predict ${Y} \sim {\rm N}_{n,m} (M, I_n, I_m)$ by a predictive density $\hat{p}({Y} \mid X)$. 
We evaluate predictive densities by the Kullback--Leibler loss:
\begin{align*}
	D ({p}(\cdot \mid M), \hat{p} (\cdot \mid X))
	= \int {p}({Y} \mid M) \log \frac{{p}({Y} \mid M)}{\hat{p} ({Y} \mid X)} {\rm d} {Y}. 
\end{align*}
The Bayesian predictive density based on a prior $\pi(M)$ is defined as
\begin{align*}
	\hat{p}_{\pi} ({Y} \mid X) & = \int {p} ({Y} \mid M) \pi (M \mid X) {\rm d} M,
\end{align*}
where $\pi(M \mid X)$ is the posterior distribution of $M$ given $X$, and it minimizes the Bayes risk \citep{aitchison1975goodness}.
Although the Bayesian predictive density with respect to the uniform prior is minimax, it is inadmissible and dominated by Bayesian predictive densities based on superharmonic priors \citep{komaki2001shrinkage,george2006improved}.

In this study, we propose a broad class of shrinkage estimators and priors by generalizing the above ones.
It is based on the (quasi-)norm of a matrix $A$ defined by
\begin{align}
\Vert A\Vert_{p}=\Big(\sum_{i=1}^m\sigma_i^p \Big)^{1/p}, \label{schatten}
\end{align}
for $p>0$, where $\sigma_1\ge\cdots\ge\sigma_m$ are the singular values of $A$. 
For $p \geq 1$, it is called the Schatten norm.
In particular, the Schatten norm with  $p=2$ coincides with the Frobenius norm and the Schatten norm with $p=1$ is called the nuclear norm, which is often used for low-rank regularization. 
Note that $\| \cdot \|_p$ with $p\in(0,1)$ is not a norm but a quasinorm.
We derive sufficient conditions for the proposed estimators to be minimax and for the proposed priors to be superharmonic.
Numerical results show their advantages compared to existing ones.

This paper is organized as follows.
In Section~\ref{sec:est}, we introduce a class of matrix norm shrinkage estimators. 
In Section~\ref{sec:prior}, we propose a class of matrix norm shrinkage priors associated with the estimators in Section~\ref{sec:est}. 
Numerical experiments are presented in Section~\ref{sec:experiment}.
Proofs of the theorems are given in the Appendix.

\section{Matrix norm shrinkage estimator}\label{sec:est}
Let $X=USV^\top$ be a singular value decomposition of $X$, where ${S}=\mathrm{diag} (s_1,\dots,s_m)$ with $s_1 \geq \dots \geq s_m$. When $M$ is the zero matrix, the probability density function (PDF) of $S$ with respect to ${\rm d}{S}$ is proportional to 
\[ \exp\left(-\frac{1}{2}\sum_{i=1}^m s_i^2\right)\prod_{i=1}^m s_i^{n-m}\prod_{i<j}(s_i^2-s_j^2), \] thus $P(s_1>s_2>\dots>s_m>0)=1$ holds \citep{james1954normal}. Using the boundedness of the ratio between the PDF of $X$ that follows $X_{ij} \sim {\rm N} (M_{ij},1)$ and the PDF of $X$ that follows $X_{ij} \sim {\rm N} (0,1)$ within $\{X \mid \| X \|_{\mathrm{F}}\le R\}$ for any given $R$, we know that $P(s_1>s_2>\dots>s_m>0)=1$ for the general case. Therefore, we assume $s_1>s_2>\dots>s_m>0$ in this paper.
We consider an equivariant estimator
\begin{align}
\hat M=U \hat{\Sigma} V^\top, \label{equiv}
\end{align}
where $\hat{\Sigma}=\mathrm{diag} (\hat{\sigma}_1(s),\dots,\hat{\sigma}_m(s))$ and $s=(s_1,\dots,s_m)$.
For convenience, we specify such an equivariant estimator by using $\hat{\sigma}_1,\dots,\hat{\sigma}_m$ in the following.
An unbiased estimate of its Frobenius risk has been derived as follows, which also appeared in \cite{tsukuma2008admissibility,candes2013unbiased}.

\begin{proposition} \citep{stein1974estimation,matsuda2019improved}\label{unbias}
    For an equivariant estimator $\hat M$ in \eqref{equiv} with $\hat{\sigma}_i=(1-\phi_i(s)) s_i$ for $i=1,\dots,m$,  its Frobenius risk is given by 
\begin{equation*}
    {\rm E}_M \left[ \| \hat{M}-M \|_{\mathrm{F}}^2 \right] = {\rm E}_M \left[ nm+\sum_{i=1}^m\Big\{s_i^2\phi_i^2-2(n-m+1)\phi_i-2s_i\frac{\partial\phi_i}{\partial s_i} \Big\}-4\sum_{i<j}\frac{s_i^2\phi_i-s_j^2\phi_j}{s_i^2-s_j^2} \right].
\end{equation*}
\end{proposition}

By extending the James--Stein and Efron--Morris estimators, we introduce the equivariant estimator $\hat{M}_{p,\alpha}$ defined by
\[
\hat{\sigma}_i=\left( 1-\alpha\frac{s_i^{p-2}}{\sum_{j=1}^ms_j^p} \right) s_i, \quad i=1,\dots,m.
\]
Note that the James--Stein estimator \eqref{JS} corresponds to $p=2$ and $\alpha=nm-2$, whereas the Efron--Morris estimator \eqref{EM} corresponds to $p=0$ and $\alpha=m(n-m-1)$. 
This estimator attains minimaxity for an appropriate choice of $p$ and $\alpha$ as follows.

\begin{theorem} \label{norm-minimax}
   For $p\in[1,2]$, $\hat M_{p,\alpha}$ is minimax if $\alpha\in[0, 2nm-(2-p)m(m+1)-2p].$ 
   For $p\in(0,1)$, $\hat M_{p,\alpha}$ is minimax if $\alpha\in[0, 2nm-2m(m+1)+2(m-1)p].$
\end{theorem}

We refer to $\hat{M}_{p,\alpha}$ with $p=1$ and $\alpha=nm-m(m+1)/2-1$ as the nuclear norm shrinkage estimator (NNSE) and denote it by $\hat{M}_{\mathrm{NNS}}$.
This estimator shrinks the singular values by a constant value that depends on the nuclear norm $\| X \|_1 = \sum_{j=1}^m s_j$ of $X$:
\begin{align}
\hat{\sigma}_i =  s_i-\frac{nm-m(m+1)/2-1}{\sum_{j=1}^m s_j}, \quad i=1,\dots,m. \label{NNSest}
\end{align} 
From Corollary 3.1 of \cite{tsukuma2008admissibility}, $\hat{M}_{\mathrm{NNS}}$ is dominated by its positive-part, which we call $\hat{M}_{\mathrm{NNS+}}$. 
Namely, $\hat{M}_{\mathrm{NNS+}}$ is defined by
\[
\hat{\sigma}_i=\max\left( s_i-\frac{nm-m(m+1)/2-1}{\sum_{j=1}^m s_j}, 0 \right), \quad i=1,\dots,m.
\]

\begin{remark}
Note that $\hat{M}_{\mathrm{NNS+}}$ has a similar form to the Singular Value Thresholding (SVT) estimator by \cite{cai2010singular}:
$\hat\sigma_i=\max( s_i-\lambda, 0)$, which is the solution of the nuclear norm regularized optimization problem:
$\min_M \{-\log p(X\mid M)+\lambda\Vert M\Vert_1\}.$
This estimator is widely used in the estimation of low-rank matrices. 
In practice, the regularization parameter $\lambda$ is usually determined by minimizing Stein's unbiased risk estimate \citep{candes2013unbiased}. 
The estimator $\hat{M}_{\mathrm{NNS+}}$ can be viewed as determining $\lambda$ based on the nuclear norm of $X$: $\lambda=(nm-m(m+1)/2-1)/\Vert X\Vert_1$.
While $\hat{M}_{\mathrm{NNS+}}$ may not work better than SVT in estimation of low-rank matrices, $\hat{M}_{\mathrm{NNS+}}$ attains minimaxity.
\end{remark}

We also develop another class of minimax estimators.
Let $\hat{M}^*_{p,\alpha}$ be an equivariant estimator defined by
\[
\hat\sigma_i= s_i-\alpha s_i^{p-1}, \quad i=1,\dots,m.
\]
When $n-m+p-1>0$, the unbiased estimate of risk in Proposition~\ref{unbias} for $\hat{M}^*_{p,\alpha}$ is
\begin{equation*}
    nm+\alpha\left\{\alpha\sum_{i=1}^m s_i^{2p-2}-2(n-m+p-1)\sum_{i=1}^m s_i^{p-2}-4\sum_{i<j}\frac{ s_i^p- s_j^p}{ s_i^2- s_j^2}\right\},
\end{equation*}
which is minimized with respect to $\alpha$ at 
\begin{equation}
\alpha( s,p)= \left( \sum_{i=1}^m s_i^{2p-2} \right)^{-1} \left( (n-m+p-1)\sum_{i=1}^m s_i^{p-2}+2\sum_{i<j}\frac{ s_i^p- s_j^p}{ s_i^2- s_j^2} \right). \label{alpha_opt}
\end{equation}
Thus, we consider the estimator $\hat{M}^*_p$ defined by
\[
\hat\sigma_i= s_i-\alpha( s,p)  s_i^{p-1}, \quad i=1,\dots,m.
\]
For $p=2$, it is approximately equal to the James--Stein estimator from $\alpha( s,2)=nm/(\sum_{i=1}^m s_i^2)$. 
For $p=0$, it coincides with the Efron--Morris estimator from $\alpha( s,0)=n-m-1$. 
This estimator attains minimaxity under certain conditions as follows.

\begin{theorem} \label{sure-minimax}
   For $p\in(0,1]$, $\hat{M}^*_p$ is minimax if $n-m\ge5-3p$. For $p\in(1,2]$, $\hat{M}^*_p$ is minimax if $n-m\ge 3-p+(4p-4)/m$.
\end{theorem}

Let $\hat{M}^{*,+}_p$ be the positive-part estimator derived from $\hat{M}^{*}_p$ given by
\[
\hat\sigma_i=\max\left( s_i-\alpha( s,p) s_i^{p-1}, 0\right), \quad i=1,\dots,m.
\]
From \cite{tsukuma2008admissibility}, $\hat{M}^{*,+}_p$ dominates $\hat{M}^{*}_p$.
Thus, from Theorem~\ref{sure-minimax}, $\hat{M}^{*,+}_p$ is also minimax for $p\in(0,1]$ and $n-m\ge5-3p.$

\section{Matrix norm shrinkage prior}\label{sec:prior}
Let $\sigma_1 \geq \cdots \geq \sigma_m \geq 0$ be the singular values of $M$.
We consider a class of priors that shrink the matrix norm \eqref{schatten} of $M$ towards zero:
\[
\pi_{p,\alpha}(M)=\Vert M\Vert_{p}^{-\alpha}=\left( \sum_{i=1}^m \sigma_i^p \right)^{-\alpha/p}.
\]
For $p=2$ and $\alpha=nm-2$, it coincides with Stein's prior \eqref{stein}. 
Also, from
$$\lim_{p\to 0}\Vert M\Vert_{p}m^{-1/p}=\lim_{p\to 0}\Big(\sum_{i=1}^m(1+p\log\sigma_i+o(p)) \Big)^{1/p}m^{-1/p}=\prod_{i=1}^m\sigma_i^{1/m}=\det(M^\top M)^{1/(2m)},$$
the singular value shrinkage prior \eqref{SVS} can be viewed as $\pi_{p,\alpha}$ for $p=0$ and $\alpha=m(n-m-1)$. 

We derive a sufficient condition for $\pi_{p,\alpha}$ to be superharmonic. 
Since $\pi_{p,\alpha}(M)$ only depends on the singular values of $M$, we use the following Laplacian formula in the singular value coordinate system. 

\begin{proposition}\citep{stein1974estimation, matsuda2019improved}\label{matrix-lapla}
If a twice weakly differentiable matrix function $f(M)$ only depends on the singular values of $M$, then its Laplacian is given by
\[
\Delta f=\sum_{i=1}^n\sum_{j=1}^m\frac{\partial^2 f}{\partial M_{ij}^2}
=2\sum_{i<j}\frac{\sigma_i\frac{\partial f}{\partial \sigma_i}-\sigma_j\frac{\partial f}{\partial \sigma_j}}{\sigma_i^2-\sigma_j^2}+(n-m)\sum_{i=1}^m\frac{1}{\sigma_i}\frac{\partial f}{\partial \sigma_i}+\sum_{i=1}^m\frac{\partial^2 f}{\partial \sigma_i^2}.
\]
\end{proposition}

\begin{theorem} \label{norm-harmonic}
For $p\in[1,2]$, $\pi_{p,\alpha}(M)$ is superharmonic if $\alpha\in[0,nm-(1-p/2)m(m+1)-p]$. 
For $p\in(0,1),$ $\pi_{p,\alpha}(M)$ is superharmonic if $\alpha\in[0,nm-m(m+1)+(m-1)p]$.
\end{theorem}

\begin{corollary}
For $p\in[1,2]$, the (generalized) Bayes estimator with respect to $\pi_{p,\alpha}(M)$ is minimax if $\alpha\in[0,nm-(1-p/2)m(m+1)-p]$. 
For $p\in(0,1),$ the (generalized) Bayes estimator with respect to $\pi_{p,\alpha}(M)$ is minimax if $\alpha\in[0,nm-m(m+1)+(m-1)p]$.
\end{corollary}

Using Theorem~\ref{norm-harmonic}, the prior $\pi_{p,\alpha}$ with $p=1$ and $\alpha=nm-m(m+1)/2-1$ is superharmonic. 
We call it a nuclear norm shrinkage (NNS) prior:
\begin{align}
\pi_{\mathrm{NNS}}(M) = \Vert M\Vert_{1}^{m(m+1)/2+1-nm}. \label{NNSprior}
\end{align}

The proposed prior works well in Bayesian prediction as well.
Suppose that we observe $X \sim {\rm N} (M, I_n, I_m)$ and predict ${Y} \sim {\rm N} (M, I_n, I_m)$ by a predictive density $\hat{p}({Y} \mid X)$. 
We evaluate predictive densities by the Kullback--Leibler loss:
\begin{align*}
	D ({p}(\cdot \mid M), \hat{p} (\cdot \mid X))
	= \int {p}({Y} \mid M) \log \frac{{p}({Y} \mid M)}{\hat{p} ({Y} \mid X)} {\rm d} {Y}. 
\end{align*}
The Bayesian predictive density based on a prior $\pi(M)$ is defined as
\begin{align*}
	\hat{p}_{\pi} ({Y} \mid X) & = \int {p} ({Y} \mid M) \pi (M \mid X) {\rm d} M,
\end{align*}
where $\pi(M \mid X)$ is the posterior distribution of $M$ given $X$, and it minimizes the Bayes risk \citep{aitchison1975goodness}.
The Bayesian predictive density with respect to the uniform prior is minimax.
However, it is inadmissible and dominated by Bayesian predictive densities based on superharmonic priors \citep{komaki2001shrinkage,george2006improved}.
Thus, we obtain the following result.

\begin{corollary}
For $p\in[1,2]$, the Bayesian predictive density with respect to $\pi_{p,\alpha}(M)$ is minimax if $\alpha\in[0,nm-(1-p/2)m(m+1)-p]$. 
For $p\in(0,1),$ the Bayesian predictive density with respect to $\pi_{p,\alpha}(M)$ is minimax if $\alpha\in[0,nm-m(m+1)+(m-1)p]$.
\end{corollary}

\section{Numerical experiments}\label{sec:experiment}

We examine the performance of the proposed estimators and priors in simulation. First, we compare the proposed estimators for their Frobenius risk. We also compare the bias of these estimators in the Appendix. Then we compare the Kullback–Leibler risk of Bayesian predictive densities based on these priors.

Figure~\ref{fig1} plots the Frobenius risk of several estimators for $m=3$ and $n=5$.
We sampled $X$ $10^5$ times and approximated the Frobenius risk by the sample mean of the Frobenius loss. 
We compare the nuclear norm shrinkage estimator (NNSE) \eqref{NNSest}, James--Stein (JS) estimator \eqref{JS}, Efron--Morris (EM) estimator \eqref{EM} and the generalized Bayes estimators based on the nuclear norm shrinkage (NNS) prior \eqref{NNSprior}, Jeffreys prior $\pi(M) \equiv 1$ (maximum likelihood estimator $\hat M=X$), Stein's prior \eqref{stein}, and the singular value shrinkage (SVS) prior \eqref{SVS}.  
Figure \ref{fig1}(a) shows the risk functions for $m = 3$, $n = 5$, $\sigma_1= 20$ and $\sigma_3 = 0$.  
Figure \ref{fig1}(b) shows the risk functions for $m = 3$, $n = 5$, $\sigma_2=\sigma_3 = 0$. 
We can see that all the estimators outperform the maximum likelihood estimator $\hat M=X$, which is also the generalized Bayes estimator with respect to the Jeffreys prior. 
The generalized Bayes estimators based on the nuclear norm shrinkage prior, Stein's prior, and the singular value shrinkage prior have a similar performance to the nuclear norm shrinkage estimator, James--Stein estimator, and Efron--Morris estimator, respectively. 
When all the singular values are close to zero, the James--Stein estimator and Stein's prior perform the best. 
When $\sigma_1$ and $\sigma_2$ are moderately large, the nuclear norm shrinkage estimator and nuclear norm shrinkage prior perform better than others. 
When $\sigma_1$ and $\sigma_2$ are large, the Efron--Morris estimator and singular value shrinkage prior perform the best.

\begin{figure}[H]
	\centering
	\includegraphics[width=\textwidth]{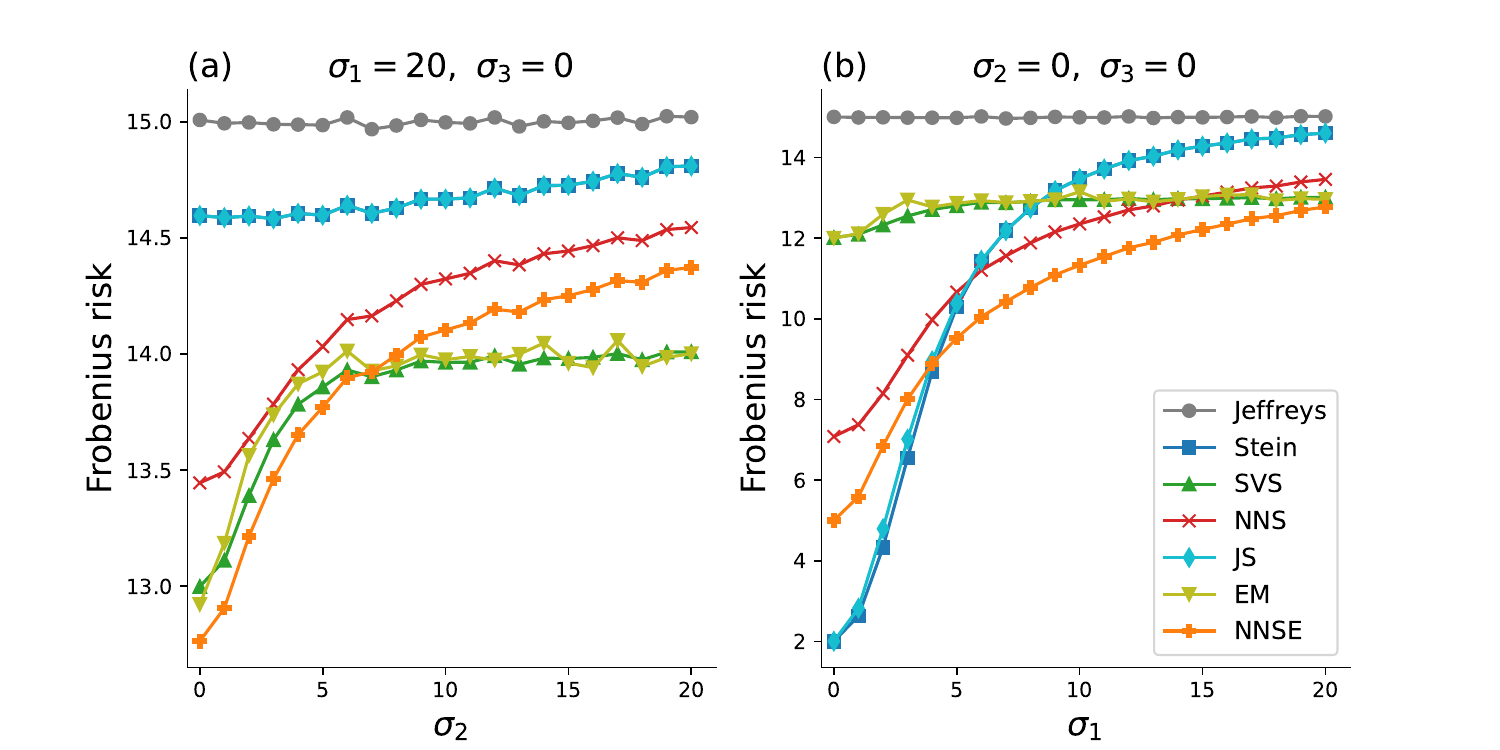}
	\caption{Frobenius risk functions of shrinkage estimators and generalized Bayes estimators when $m = 3$, $n = 5$.}
        \label{fig1}
\end{figure}

Figure~\ref{fig2} compares the shrinkage estimators discussed in Section~\ref{sec:est}. 
Here, we only consider positive-part estimators because they dominate the corresponding shrinkage estimators. 
We also consider the positive-part modified Efron--Morris (mEM+) estimator proposed by \cite{efron1976multivariate}:
\[
\hat\sigma_i=\max \left(s_i-\frac{n-m-1}{s_i}-(m-1)(m+2) \frac{s_i}{\sum_{j=1}^m s_j^2}, 0 \right), \quad i=1,\dots,m.
\]
Figure~\ref{fig2} plots the risk functions with respect to $\sigma_1$ when $m = 20$, $n = 100$, $\sigma_i=(6-i)\sigma_1/5$ for
$i=2,\dots,5$, and $\sigma_6=\cdots=\sigma_{20}=0$. 
When $\sigma_1$ is less than $10$, JS+ estimator performs the best. 
When $\sigma_1$ is in $[15,25]$, $\hat M^{*,+}_{1}$ attains the best performance. 
When $\sigma_1$ is larger than $30$, $\hat M^{*,+}_{\frac{1}{2}}$ outperforms other estimators.

\begin{figure}[H]
	\centering
	\includegraphics[width=.8\textwidth]{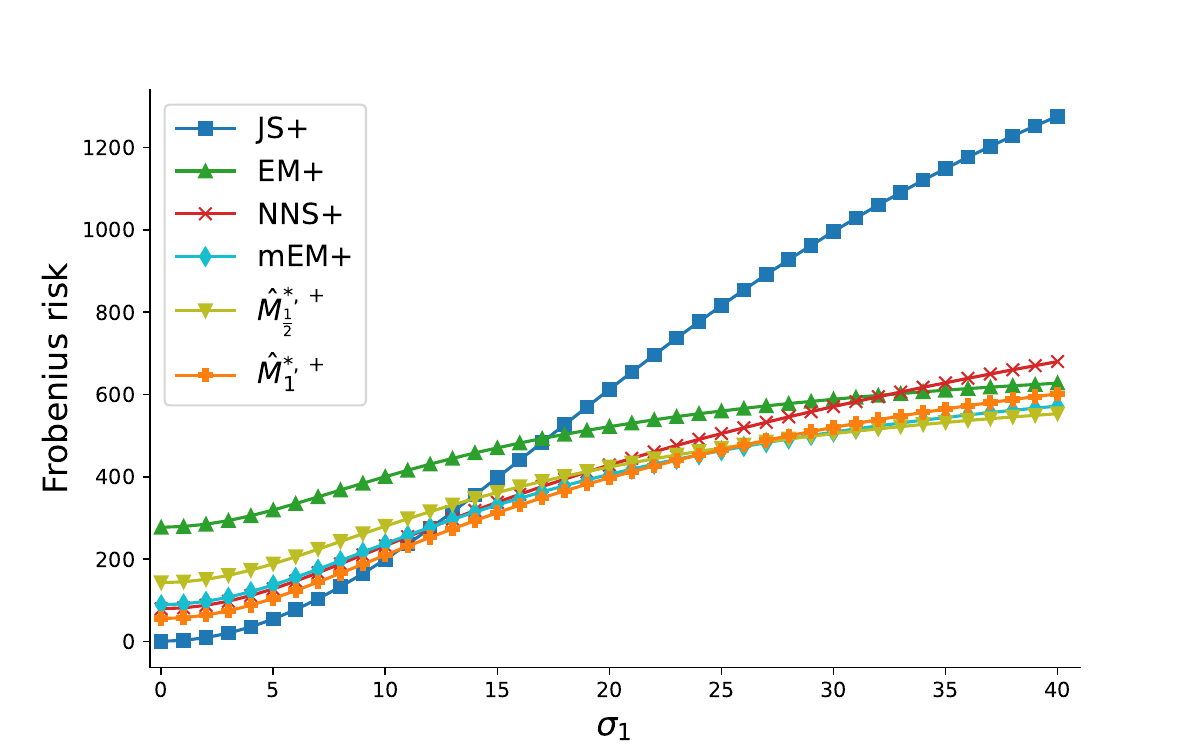}
	\caption{Frobenius risk functions of shrinkage estimators when $m = 20$, $n = 100$, $\sigma_i=(6-i)\sigma_1/5$ for
$i=2,\dots,5$, and $\sigma_6=\cdots=\sigma_{20}=0$.}
        \label{fig2}
\end{figure}

Table~\ref{frob_risk} presents the Frobenius risk on two types of low rank matrices. 
When $\sigma_1$ is small, JS+ has the best performance. 
When $\sigma_1$ is large, $\hat{M}^{*,+}_{\frac{1}{2}}$ has the best performance. 
It can be seen that $\hat{M}^{*,+}_{\frac{1}{2}}$ outperforms EM+ in all cases. 
Although the two estimators $\hat{M}^{*,+}_{1}$ and NNS+ have a similar form to the SVT estimator, $\hat{M}^{*,+}_{1}$ outperforms NNS+ in all cases.

\begin{table}[h]
\centering
\caption{Frobenius risk of shrinkage estimators}
\label{frob_risk}
\begin{center}
\begin{tabular}{c c c c c c c}
 \hline
 & mEM+ &EM+ & $\hat{M}^{*,+}_{\frac{1}{2}}$ & NNS+ & $\hat{M}^{*,+}_{1}$ & JS+\\
 \hline
 \multicolumn{7}{l}{$n=20$, $m=15$, $\sigma_1=\cdots=\sigma_{10}$, $\sigma_{11}=\dots=\sigma_{15}=0$} \\
 \hline
 $\sigma_1=1$ &13.6 & 215.6 & 120.0 & 75.0 & 44.1 & 10.8*\\ 
$\sigma_1=10$ & 223.6* & 268.7 & 243.9 & 239.4 & 232.1 &231.5\\
 $\sigma_1=100$ & 273.4 & 274.0 & 263.6* & 294.7 & 282.5 &299.1\\
 $\sigma_1=1000$ & 274.1 & 274.1 & 263.0* & 299.5 & 286.2 & 300.0\\
 \hline
 \multicolumn{7}{l}{$n=100$, $m=20$, $\sigma_i=\sigma_1\max(1-(i-1)/5,0)$}\\
 \hline
 $\sigma_1=1$ &91.7& 279.8 & 145.3 & 82.1 & 58.2 & 3.2*\\ 
$\sigma_1=10$ & 238.2 & 400.8 & 280.6 & 233.0 & 210.2 &200.0*\\
 $\sigma_1=100$ & 691.6 & 703.6 & 646.7* & 1085.5 & 801.5 &1833.5\\
 $\sigma_1=1000$ & 728.8 & 728.9 & 662.2* & 1842.2 & 958.0 & 1998.2\\
 $\sigma_1=10000$ & 729.2 & 729.2 & 656.1* & 1983.1 & 974.5 & 2000.0\\
 \hline
\end{tabular}
\end{center}
\end{table}

Finally, we consider the prediction problem, where we predict $Y\sim \mathrm{N}(M,I_n, I_m)$ using $X\sim \mathrm{N}(M,I_n, I_m)$. 
We compare the Kullback--Leibler risk  of the Bayesian predictive density $p_{\pi}(Y\mid X)$ based on the Jeffreys prior $\pi(M) \equiv 1$, Stein's prior \eqref{stein}, singular value shrinkage prior \eqref{SVS}, and the nuclear norm shrinkage prior \eqref{NNSprior}. 
We sampled $(X, Y)$ $10^5$ times and
approximated the Kullback--Leibler risk by the sample mean of $\log p(Y\mid M)-\log p_{\pi}(Y\mid X)$.  
We set $m = 3$, $n = 5$, and $\sigma_3=0$. Figure \ref{fig3}(a) shows the risk functions for $\sigma_1=20$.  Figure \ref{fig3}(b) shows the risk functions for $\sigma_2= 0$. We can see that all the priors outperform the Jeffreys prior. When all the singular values are close to zeros, Stein prior performs the best. When $\sigma_1$ and $\sigma_2$ are moderately large, the nuclear norm shrinkage prior performs the best. When $\sigma_1$ and $\sigma_2$ are large, the singular value shrinkage prior performs the best.

\begin{figure}[H]
	\centering
	\includegraphics[width=\textwidth]{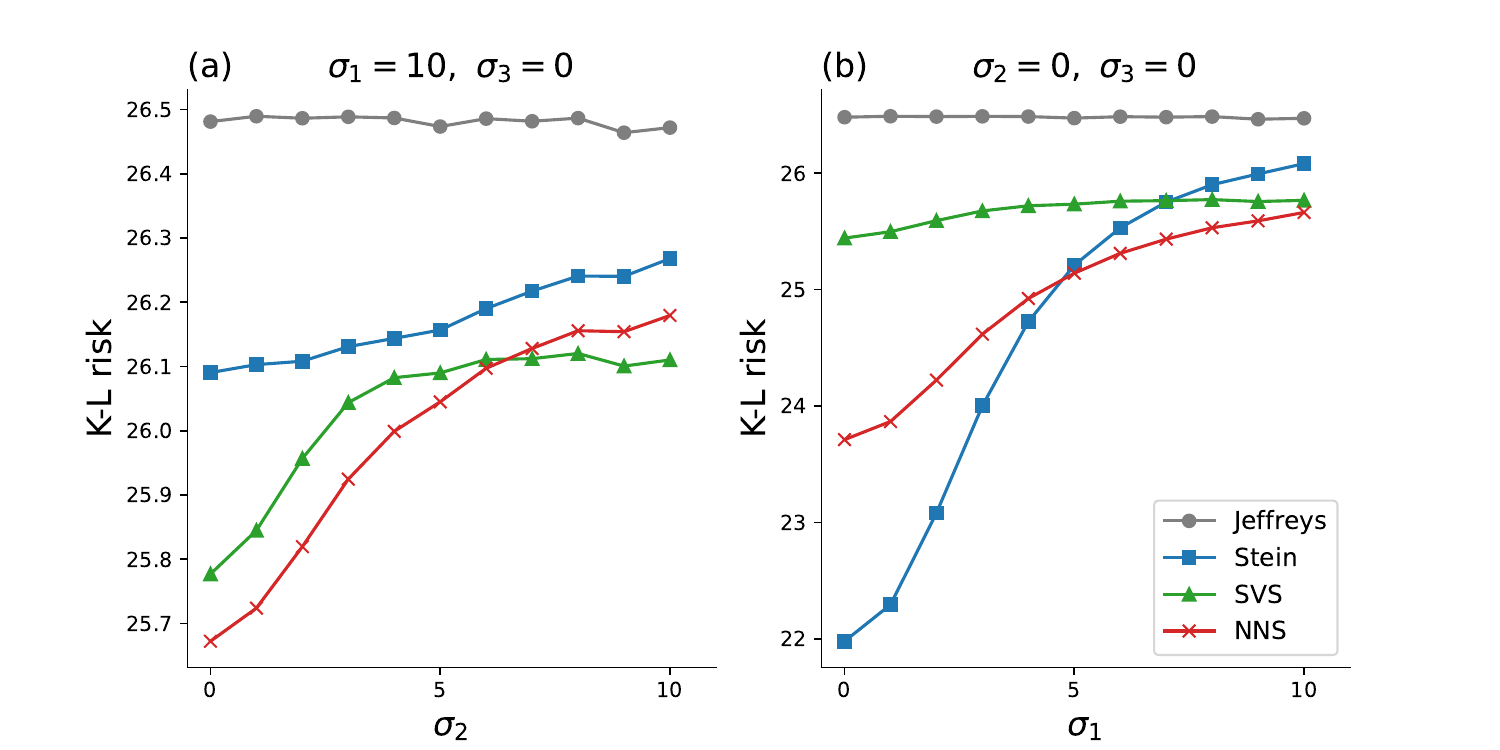}
	\caption{Kullback--Leibler risk functions of Bayesian predictive densities when $m = 3$, $n = 5$.}
        \label{fig3}
\end{figure}

\section{Discussion}

In this paper, we proposed a class of minimax estimators for a normal mean matrix and a class of matrix norm shrinkage priors. We confirmed the effectiveness of these estimators and priors by numerical results.

In the numerical experiment presented in Figure~\ref{fig1}, the JS estimator and the generalized Bayes estimator based on Stein's prior exhibit similar performance, as do the EM estimator and SVS prior. However, the NNSE estimator and NNS prior display certain differences. The differences in Frobenius risk among these estimators require future study. Note that we only proved the minimaxity of the proposed estimator $\hat{M}^*_p$ when the parameter $p$ is fixed. It is a future problem to develop minimax estimators that dynamically select $p$ based on the observation $X$.

\section*{Acknowledgments}
The authors thank the referees and Associate Editor for their constructive and insightful comments that helped improve this manuscript.
Takeru Matsuda was supported by JSPS KAKENHI Grant Numbers 22K17865, 24K02951, 25K22750, and 26H02517, and JST Grant Numbers JPMJMS2024 and JPMJAP25B1.
Fumiyasu Komaki was supported by JSPS KAKENHI Grant Numbers 22H00510 and 25H01156.

\appendix
\section{Comparison of the Bias of Proposed Estimators}

Figure~\ref{fig4} plots the Frobenius norm of bias of several estimators for $m=3$ and $n=5$.
We sampled $X$ $10^5$ times and approximated ${\rm E}_M[\hat{M}]$ by the sample mean of the estimator. Then Frobenius norm of bias is defined by $ \| {\rm E}_M[\hat{M}]-M \|_{\mathrm{F}} $.
We compare the nuclear norm shrinkage estimator (NNSE) \eqref{NNSest}, James--Stein (JS) estimator \eqref{JS}, Efron--Morris (EM) estimator \eqref{EM} and the generalized Bayes estimators based on the nuclear norm shrinkage (NNS) prior \eqref{NNSprior}, Jeffreys prior $\pi(M) \equiv 1$ (maximum likelihood estimator $\hat M=X$), Stein's prior \eqref{stein}, and the singular value shrinkage (SVS) prior \eqref{SVS}.  
Figure \ref{fig4}(a) shows the risk functions for $m = 3$, $n = 5$, $\sigma_1= 20$ and $\sigma_3 = 0$.  
Figure \ref{fig4}(b) shows the risk functions for $m = 3$, $n = 5$, $\sigma_2=\sigma_3 = 0$. The maximum likelihood estimator $\hat M=X$ is unbiased. Among the other estimators, the Frobenius norm of bias of the generalized Bayes estimator based on the SVS prior is the smallest, while that of the JS estimator is the largest.

\begin{figure}[H]
	\centering
	\includegraphics[width=\textwidth]{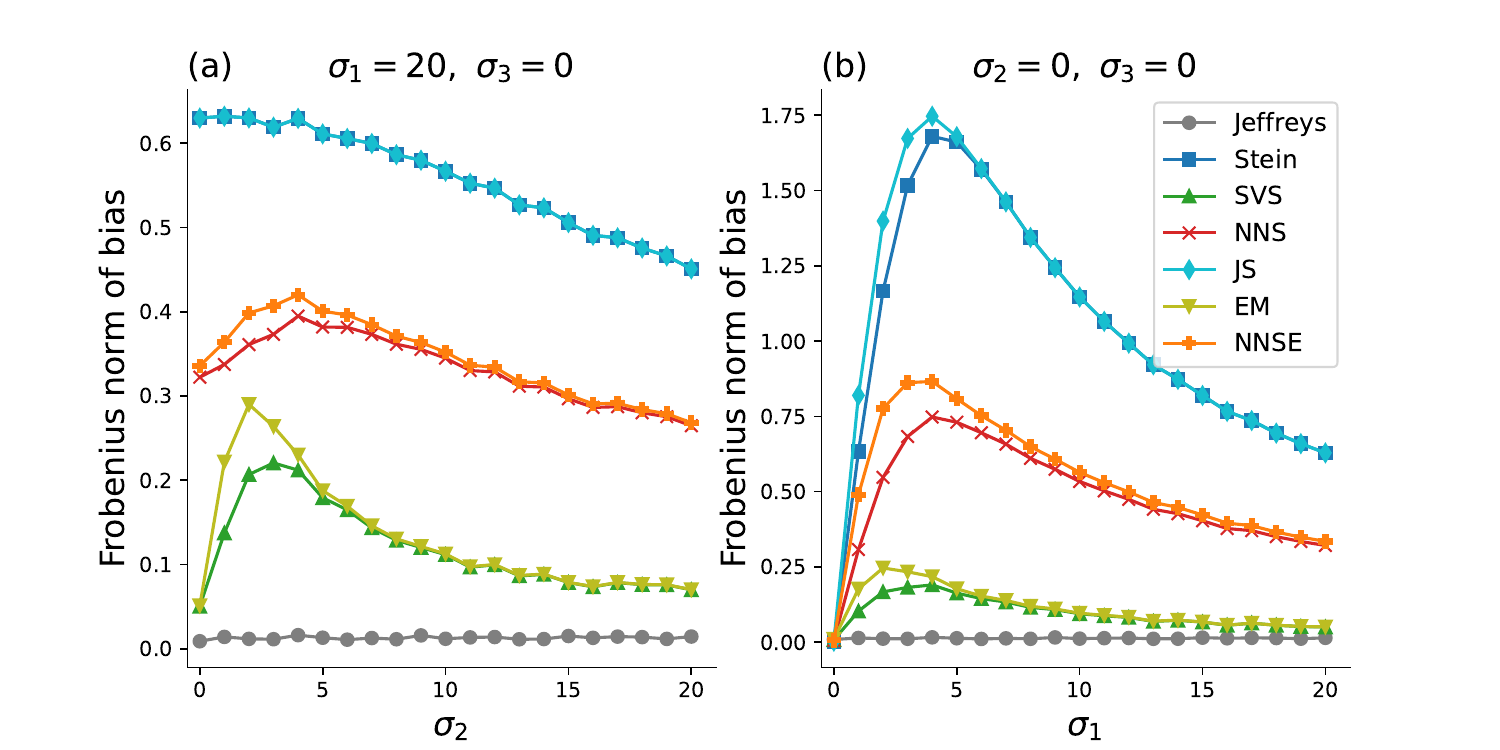}
	\caption{Frobenius norm of bias of shrinkage estimators and generalized Bayes estimators when $m = 3$, $n = 5$.}
        \label{fig4}
\end{figure}

Figure~\ref{fig5} compares several shrinkage estimators for $m=20$ and $n=100$. The Frobenius norm of bias of EM+ is the smallest, while that of $\hat{M}^{*,+}_{\frac{1}{2}}$ is the second smallest.

\begin{figure}[H]
	\centering
	\includegraphics[width=.8\textwidth]{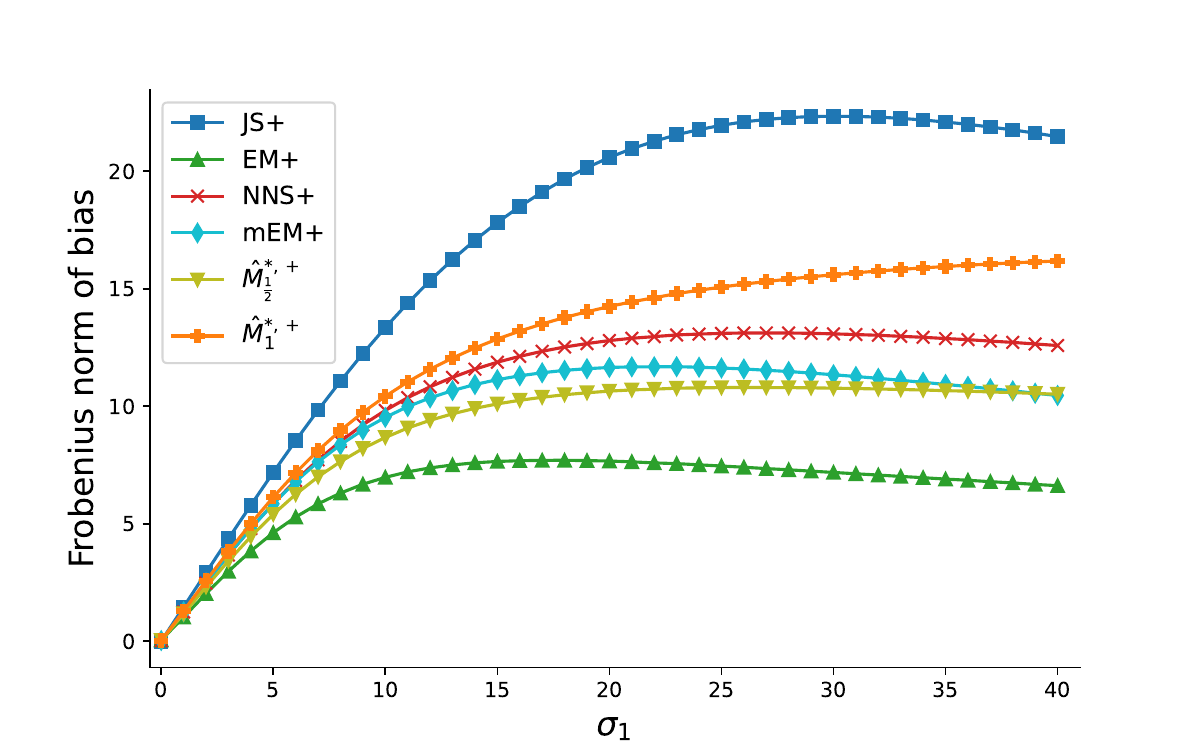}
	\caption{Frobenius norm of bias of shrinkage estimators when $m = 20$, $n = 100$, $\sigma_i=(6-i)\sigma_1/5$ for
$i=2,\dots,5$, and $\sigma_6=\cdots=\sigma_{20}=0$.}
        \label{fig5}
\end{figure}

\section{Technical Lemmas}

\begin{lemma} \label{nns-dec}
When $p\in(0,2]$,
$(x^p-1)/(x^2-1)$ is a decreasing function of $x>1$ and a decreasing function of $x\in(0,1)$. Moreover, $(x^p-1)/(x^2-1)<p/2$ when $x>1$ and $(x^p-1)/(x^2-1)>p/2$ when $x\in(0,1).$
\end{lemma}
\begin{proof}
    The differential function is 
    $$\frac{(x^p-1)'(x^2-1)-(x^p-1)(x^2-1)'}{(x^2-1)^2}=\frac{(p-2)x^{p+1}-px^{p-1}+2x}{(x^2-1)^2}.$$
    We define $g(x)=(p-2)x^p-px^{p-2}+2$. Then the  differential function is $xg(x)/(x^2-1)^2$. Note that $g(1)=0$ and $$g'(x)=(p-2)px^{p-1}-p(p-2)x^{p-3}=(2-p)px^{p-3}(1-x^2).$$ Thus, $g(x)\le0$.

    Moreover, $\lim_{x\to1}(x^p-1)/(x^2-1)=p/2$. Therefore, $(x^p-1)/(x^2-1)<p/2$ when $x>1$ and $(x^p-1)/(x^2-1)>p/2$ when $x\in(0,1).$
\end{proof}

\begin{lemma} \label{bernoulli}
When $p\in[1,2]$ and $\sigma_i,\ \sigma_j$ are different nonnegative numbers,
\begin{equation}
        \frac{\sigma_i^{2p}-\sigma_j^{2p}}{\sigma_i^{2}-\sigma_j^2}\ge \frac{p}{2}(\sigma_i^{2p-2}+\sigma_j^{2p-2}). \label{lemma2-1}
    \end{equation}
\end{lemma}
\begin{proof}
    When $ \sigma_i > \sigma_j $, we set $ x = \sigma_i/\sigma_j $; otherwise, we set $ x = \sigma_j/\sigma_i $. Dividing both sides of \eqref{lemma2-1} by 
$\min(\sigma_i, \sigma_j)^{2p-2}$, we need only show that $$\frac{x^{2p}-1}{x^2-1}\ge \frac{p}{2}(x^{2p-2}+1)$$ for any $x> 1$. Define $$g(x)=x^{2p}-1- p/2(x^{2p-2}+1)(x^2-1),$$ we need only show that $g(x)\ge0$ for any $x> 1$. Note $$g'(x)=(2-p)px^{2p-1}-px+p(p-1)x^{2p-3}=:pxh(x),$$ where $h(x)=(2-p)x^{2p-2}-1+(p-1)x^{2p-4}$. Note that $$h'(x)=(2-p)(2p-2)x^{2p-3}-(p-1)(4-2p)x^{2p-5}=(2-p)(2p-2)(x^{2p-3}-x^{2p-5})\ge0.$$ Thus, $h(x)\ge h(1)=0$. Thus, $g'(x)=pxh(x)\ge0$. Hence $g(x)\ge g(1)=0.$
\end{proof}

\begin{lemma} \label{chebyshev}
(Chebyshev's sum inequality)
If $a_1\ge a_2\ge\cdots\ge a_m$ and $b_1\le b_2\le\cdots\le b_m$, then $$\Big(\sum_{i=1}^m a_i\Big)\Big(\sum_{i=1}^m b_i\Big)\ge m\sum_{i=1}^m a_ib_i.$$
\end{lemma}

\begin{lemma}\label{chebyshev-2}
When $p\in(0,2]$ and $\sigma_1,\dots,\sigma_m$ are different nonnegative numbers,
\begin{equation}
\Big(\sum_{i=1}^m\sigma_i^p\Big)\sum_{i<j}\frac{\sigma_i^p-\sigma_j^p}{\sigma_i^2-\sigma_j^2}\ge\frac{m}{2}\sum_{i<j}\frac{\sigma_i^{2p}-\sigma_j^{2p}}{\sigma_i^2-\sigma_j^2}. \label{lemma4-1}
\end{equation}
\end{lemma}
\begin{proof}
Since swapping the values of $\sigma_1, \ldots, \sigma_m$ arbitrarily does not affect the value on either side of \eqref{lemma4-1}, we can assume without loss of generality that $\sigma_1>\cdots>\sigma_m$.
For given $i$, using Lemma~\ref{nns-dec}, $$\frac{\sigma_i^p-\sigma_j^p}{\sigma_i^2-\sigma_j^2}=\sigma_i^{p-2}\frac{(\sigma_j/\sigma_i)^p-1}{(\sigma_j/\sigma_i)^2-1}$$ is a decreasing function of $\sigma_j$. Therefore, using $\sigma_1>\cdots>\sigma_m$, $(\sigma_i^p-\sigma_j^p)/(\sigma_i^2-\sigma_j^2)$ is an increasing function of $j$. For given $i$, $\sigma_i^p+\sigma_j^p$ is a decreasing function of $j$. Therefore, for given $i$, using Chebyshev's sum inequality, 
\begin{equation}
    \Big(\sum_{j\neq i}(\sigma_i^p+\sigma_j^p)\Big)\Big(\sum_{j\neq i}\frac{\sigma_i^p-\sigma_j^p}{\sigma_i^2-\sigma_j^2}\Big)\ge (m-1)\sum_{j\neq i}\frac{\sigma_i^{2p}-\sigma_j^{2p}}{\sigma_i^{2}-\sigma_j^2}. \label{cheb-1}
\end{equation} 
For given $j$, using Lemma~\ref{nns-dec}, $(\sigma_i^p-\sigma_j^p)/(\sigma_i^2-\sigma_j^2)$ is a decreasing function of $\sigma_i$. Therefore, if $\sigma_{i_1}\ge\sigma_{i_2}$, $$\sum_{j\neq i_1}(\sigma_{i_1}^p-\sigma_j^p)/(\sigma_{i_1}^2-\sigma_j^2)-\sum_{j\neq i_2}(\sigma_{i_2}^p-\sigma_j^p)/(\sigma_{i_2}^2-\sigma_j^2)$$ is nonpositive. Therefore, using $\sigma_1>\cdots>\sigma_m$, $\sum_{j\neq i}(\sigma_i^p-\sigma_j^p)/(\sigma_i^2-\sigma_j^2)$ is an increasing function of $i$. In addition, $\sum_{j\neq i}(\sigma_i^p+\sigma_j^p)=\sum_{j}\sigma_j^p+(m-2)\sigma_i^p$ is a decreasing function of $i$. Therefore, using Chebyshev's sum inequality,
\begin{equation}
    \Big(\sum_{i=1}^m\sum_{j\neq i}(\sigma_i^p+\sigma_j^p)\Big)\Big(\sum_{i=1}^m\sum_{j\neq i}\frac{\sigma_i^p-\sigma_j^p}{\sigma_i^2-\sigma_j^2}\Big) \ge m\sum_{i=1}^m\Big(\sum_{j\neq i}(\sigma_i^p+\sigma_j^p)\Big)\Big(\sum_{j\neq i}\frac{\sigma_i^p-\sigma_j^p}{\sigma_i^2-\sigma_j^2}\Big). \label{cheb-2}
\end{equation}
Combining \eqref{cheb-1} and \eqref{cheb-2}, we have
\begin{align*}
  \Big(\sum_{i=1}^m\sigma_i^p\Big)\sum_{i<j}\frac{\sigma_i^p-\sigma_j^p}{\sigma_i^2-\sigma_j^2}&= \frac{1}{4(m-1)}\Big(\sum_{i=1}^m\sum_{j\neq i}(\sigma_i^p+\sigma_j^p)\Big)\Big(\sum_{i=1}^m\sum_{j\neq i}\frac{\sigma_i^p-\sigma_j^p}{\sigma_i^2-\sigma_j^2}\Big) \notag\\
  &\ge \frac{m}{4(m-1)}\sum_{i=1}^m\Big(\sum_{j\neq i}(\sigma_i^p+\sigma_j^p)\Big)\Big(\sum_{j\neq i}\frac{\sigma_i^p-\sigma_j^p}{\sigma_i^2-\sigma_j^2}\Big)\notag\\
  &\ge \frac{m}{4}\sum_{i=1}^m\sum_{j\neq i}\frac{\sigma_i^{2p}-\sigma_j^{2p}}{\sigma_i^{2}-\sigma_j^2}=\frac{m}{2}\sum_{i<j}\frac{\sigma_i^{2p}-\sigma_j^{2p}}{\sigma_i^2-\sigma_j^2}.
\end{align*} 
\end{proof}

\begin{lemma} \label{nns-inq}
When $p\in[1,2]$ and $\sigma_1,\dots,\sigma_m$ are different nonnegative numbers, $$\Big(\sum_{i=1}^m\sigma_i^p\Big)\sum_{i<j}\frac{\sigma_i^p-\sigma_j^p}{\sigma_i^2-\sigma_j^2}\ge\frac{pm(m-1)}{4}\sum_{i=1}^m\sigma_i^{2p-2}.$$
\end{lemma}
\begin{proof}
Combining Lemmas \ref{bernoulli} and \ref{chebyshev-2}, we have
\begin{align*}
  \Big(\sum_{i=1}^m\sigma_i^p\Big)\sum_{i<j}\frac{\sigma_i^p-\sigma_j^p}{\sigma_i^2-\sigma_j^2}&\ge \frac{m}{2}\sum_{i<j}\frac{\sigma_i^{2p}-\sigma_j^{2p}}{\sigma_i^2-\sigma_j^2} \notag\\
  &\ge \frac{pm}{4}\sum_{i<j}(\sigma_i^{2p-2}+\sigma_j^{2p-2})=\frac{pm(m-1)}{4}\sum_{i=1}^m\sigma_i^{2p-2}.
\end{align*} 
\end{proof}

\section{Proofs of the Theorems}

\subsection{Proof of Theorem~\ref{norm-minimax}}
\begin{proof}
    Let $A=\sum_{i=1}^ms_i^p$, $B=\sum_{i=1}^ms_i^{p-2}$, and $C=\sum_{i=1}^ms_i^{2p-2}$. 
    Substituting $\phi_i=\alpha s_i^{p-2}/(\sum_{j=1}^ms_j^p)$ into the unbiased estimate of risk in Proposition \ref{unbias} yields
\begin{align}
   &nm+\sum_{i=1}^m\Big\{\frac{\alpha^2 s_i^{2p-2}}{A^2}-\frac{2(n-m+1)\alpha s_i^{p-2}}{A}+\frac{2p\alpha s_i^{2p-2}}{A^2}-\frac{2(p-2)\alpha s_i^{p-2}}{A} \Big\}-4\alpha\sum_{i<j}\frac{ s_i^p- s_j^p}{( s_i^2- s_j^2)A} \notag\\
   &=nm+\frac{\alpha}{A^2}\Big\{(\alpha+2p) C-2(n-m+p-1)AB-4A\sum_{i<j}\frac{ s_i^p- s_j^p}{ s_i^2- s_j^2}\Big\}.\label{loss-1}
\end{align}
   Thus, if \eqref{loss-1} is not larger than $nm$ for every $X$, the estimator $\hat{M}_{p,\alpha}$ is minimax.

First, assume $p\in[1,2]$.
Since $AB\ge mC$ from Lemma~\ref{chebyshev} and
$$A\sum_{i<j}\frac{ s_i^p- s_j^p}{ s_i^2- s_j^2}\ge\frac{pm(m-1)}{4}C$$
from Lemma \ref{nns-inq}, the unbiased estimate of risk \eqref{loss-1} is bounded from above by
\begin{equation}
    nm+\frac{\alpha C}{A^2}\Big\{\alpha+2p-2(n-m+p-1)m-pm(m-1)\Big\}, \label{loss-2}
\end{equation}
which is not larger than $nm$ when $\alpha\le 2nm-(2-p)m(m+1)-2p$.

Next, assume $p\in(0,1)$.
From $A>0$ and $AB\ge mC$ from Lemma~\ref{chebyshev}, the unbiased estimate of risk \eqref{loss-1} is bounded from above by
\begin{equation}
    nm+\frac{\alpha C}{A^2}\Big\{\alpha+2p-2(n-m+p-1)m\Big\}, \label{loss-3}
\end{equation}
which is not greater than $nm$ when $\alpha\le 2nm-2m(m+1)+2(m-1)p.$
\end{proof}

\begin{remark}
When $p\in[1,2]$, the upper bound \eqref{loss-2} of the unbiased risk estimate \eqref{loss-1} is minimized at $\alpha=nm-(1-p/2)m(m+1)-p.$ 
Thus, this value of $\alpha$ can be considered as a default choice. 
For $p=1$, it is given by $\alpha=nm-m(m+1)/2-1$.
\end{remark}

\subsection{Proof of Theorem~\ref{sure-minimax}}

\begin{proof}
In general, the unbiased estimate of risk in Proposition~\ref{unbias} for the equivariant estimator
\[
\hat\sigma_i= s_i-\alpha( s) s_i^{p-1}, \quad i=1,\dots,m.
\]
is
\begin{equation*}
    nm+\sum_{i=1}^m\Big\{ s_i^{2p-2}\alpha( s)^2-2(n-m+p-1) s_i^{p-2}\alpha( s)-2 s_i^{p-1}\frac{\partial\alpha( s)}{\partial s_i} \Big\}-4\sum_{i<j}\alpha( s)\frac{ s_i^p- s_j^p}{ s_i^2- s_j^2}. 
\end{equation*}
   By substituting \eqref{alpha_opt}, the unbiased estimate is equal to
   \begin{align}
   &nm+\alpha( s,p)\Big\{\alpha( s,p)\sum_{i=1}^m s_i^{2p-2}-2(n-m+p-1)\sum_{i=1}^m s_i^{p-2}-4\sum_{i<j}\frac{ s_i^p- s_j^p}{ s_i^2- s_j^2}\Big\}-2\sum_{i=1}^m s_i^{p-1}\frac{\partial\alpha( s,p)}{\partial s_i} \notag\\
    &=nm-\Big\{(n-m+p-1)\sum_{i=1}^m s_i^{p-2}+2\sum_{i<j}\frac{ s_i^p- s_j^p}{ s_i^2- s_j^2}\Big\}^2/(\sum_{i=1}^m s_i^{2p-2})-2\sum_{i=1}^m s_i^{p-1}\frac{\partial\alpha( s,p)}{\partial s_i}. \label{sure-min}
\end{align} 
When $p\le1,$
\begin{align}
    \sum_{i=1}^m s_i^{p-1}\frac{\partial\alpha( s,p)}{\partial s_i}&\ge\sum_{i=1}^m s_i^{p-1}\frac{(n-m+p-1)(p-2) s_i^{p-3}+2\sum_{j\neq i}\frac{(p-2) s_i^{p+1}-p s_i^{p-1} s_j^2+2 s_i s_j^p}{( s_i^2- s_j^2)^2}}{\sum_{i=1}^m s_i^{2p-2}}\notag\\
    &=\frac{(n-m+p-1)(p-2)\sum_{i=1}^m s_i^{2p-4}+2\sum_{i<j}\frac{(p-2)( s_i^{2p}+ s_j^{2p})-p s_i^{2p-2} s_j^2-p s_j^{2p-2} s_i^2+4 s_i^p s_j^p}{ s_i^4+ s_j^4-2 s_i^2 s_j^2}}{\sum_{i=1}^m s_i^{2p-2}} \label{sure-inq1}
    \end{align}
Also, from $ s_i^{2p}+ s_j^{2p}\le( s_i^4+ s_j^4) s_i^{p-2} s_j^{p-2}$ and $ s_i^{2p-2} s_j^2+ s_j^{2p-2} s_i^2\le( s_i^4+ s_j^4) s_i^{p-2} s_j^{p-2}$, 
\begin{align}
    \frac{(p-2)( s_i^{2p}+ s_j^{2p})-p s_i^{2p-2} s_j^2-p s_j^{2p-2} s_i^2+4 s_i^p s_j^p}{ s_i^4+ s_j^4-2 s_i^2 s_j^2}\ge-2 s_i^{p-2} s_j^{p-2}.\label{sure-inq3}
\end{align}
Therefore, using \eqref{sure-inq1}, \eqref{sure-inq3}, and $(n-m+p-1)(p-2)\le-2$,
    \begin{align}
    \sum_{i=1}^m s_i^{p-1}\frac{\partial\alpha( s,p)}{\partial s_i}&\ge\frac{(n-m+p-1)(p-2)\sum_{i=1}^m s_i^{2p-4}+2\sum_{i<j}(-2) s_i^{p-2} s_j^{p-2}}{\sum_{i=1}^m s_i^{2p-2}}\notag\\
    &\ge (n-m+p-1)(p-2)(\sum_{i=1}^m s_i^{p-2})^2/(\sum_{i=1}^m s_i^{2p-2}). \label{sure-inq2}
    \end{align}
From \eqref{sure-inq2}, the unbiased loss estimator \eqref{sure-min} is bounded from above by
\begin{align*}
   &nm-\Big\{(n-m+p-1)\sum_{i=1}^m s_i^{p-2}\Big\}^2/(\sum_{i=1}^m s_i^{2p-2})-2\sum_{i=1}^m s_i^{p-1}\frac{\partial\alpha( s,p)}{\partial s_i} \\
   \le\ & nm-\Big\{(n-m+p-1)\sum_{i=1}^m s_i^{p-2}\Big\}^2/(\sum_{i=1}^m s_i^{2p-2})-2(n-m+p-1)(p-2)(\sum_{i=1}^m s_i^{p-2})^2/(\sum_{i=1}^m s_i^{2p-2}) \\
   =\ & nm-\big[(n-m+p-1)^2+2(n-m+p-1)(p-2)\big](\sum_{i=1}^m s_i^{p-2})^2/(\sum_{i=1}^m s_i^{2p-2}).
\end{align*} 
It is smaller than $nm$ if $(n-m+p-1)^2+2(n-m+p-1)(p-2)\ge0$, which holds when $n-m\ge 5-3p$.

Next, we prove the case where $p\in(1,2]$ and $n-m\ge 3-p+(4p-4)/m$.
\begin{align}
    \sum_{i=1}^m s_i^{p-1}\frac{\partial\alpha( s,p)}{\partial s_i}\notag &=\frac{(n-m+p-1)(p-2)\sum_{i=1}^m s_i^{2p-4}+2\sum_{i<j}\frac{(p-2)( s_i^{2p}+ s_j^{2p})-p s_i^{2p-2} s_j^2-p s_j^{2p-2} s_i^2+4 s_i^p s_j^p}{ s_i^4+ s_j^4-2 s_i^2 s_j^2}}{\sum_{i=1}^m s_i^{2p-2}}\notag \\
    &+\frac{(2-2p)(\sum_{i=1}^m s_i^{3p-4})\left\{(n-m+p-1)(\sum_{i=1}^m s_i^{p-2})+2\sum_{i<j}\frac{ s_i^p- s_j^p}{ s_i^2- s_j^2}\right\}}{(\sum_{i=1}^m s_i^{2p-2})^2}. \label{sure-inq4}
    \end{align}
We abbreviate $(n-m+p-1)(\sum_{i=1}^m s_i^{p-2})+2\sum_{i<j}\frac{ s_i^p- s_j^p}{ s_i^2- s_j^2}$ as $D$. Using Lemma \ref{chebyshev}, we have $m\sum_{i=1}^m s_i^{3p-4}\le (\sum_{i=1}^m s_i^{2p-2})(\sum_{i=1}^m s_i^{p-2})$. By combining it with \eqref{sure-inq3}, \eqref{sure-inq4}, and $p>1$, we have
\begin{align}
    \sum_{i=1}^m s_i^{p-1}\frac{\partial\alpha( s,p)}{\partial s_i}\ge\frac{(n-m+p-1)(p-2)\sum_{i=1}^m s_i^{2p-4}+2\sum_{i<j}(-
    2) s_i^{p-2} s_j^{p-2}}{\sum_{i=1}^m s_i^{2p-2}}
    +\frac{(2-2p)(\sum_{i=1}^m s_i^{p-2})D}{m\sum_{i=1}^m s_i^{2p-2}}. \label{sure-inq5}
    \end{align}
Combining \eqref{sure-inq5}, $p-2>-1$, $n-m+p-1\ge2$ and $\sum_{i=1}^m s_i^{p-2}\le D/(n-m+p-1)$, we have
\begin{align}
    \sum_{i=1}^m s_i^{p-1}\frac{\partial\alpha( s,p)}{\partial s_i} \ge\frac{(m-n+1-p)(\sum_{i=1}^m s_i^{p-2})^2
    +\frac{2-2p}{m(n-m+p-1)}D^2}{\sum_{i=1}^m s_i^{2p-2}}. \label{sure-inq6}
    \end{align}
Because $n-m\ge 3-p+(4p-4)/m$, $4-4p\ge-m(n-m+p-1)$. From \eqref{sure-inq6} and $4-4p\ge-m(n-m+p-1)$, the unbiased loss estimator \eqref{sure-min} is bounded from above by
\begin{align*}
   &nm-\left(1+\frac{4-4p}{m(n-m+p-1)}\right)D^2/(\sum_{i=1}^m s_i^{2p-2})-2\frac{(m-n+1-p)(\sum_{i=1}^m s_i^{p-2})^2}{\sum_{i=1}^m s_i^{2p-2}} \\
   \le\ & nm-\left(1+\frac{4-4p}{m(n-m+p-1)}-\frac{2}{n-m+p-1}\right)\Big\{(n-m+p-1)\sum_{i=1}^m s_i^{p-2}\Big\}^2/(\sum_{i=1}^m s_i^{2p-2}) .
\end{align*}
It is smaller than $nm$ if $(n-m+p-1)+(4-4p)/m-2\ge0$, which holds when $n-m\ge 3-p+(4p-4)/m$.
\end{proof}

\subsection{Proof of Theorem~\ref{norm-harmonic}}

\begin{proof}
Let $\beta=-\alpha/p$. We consider the $C^2$ matrix function $$f_{k}(M)=\Big(\sum_{i=1}^m(\sigma_i^2+1/k)^{\frac{p}{2}} \Big)^{\beta}.$$
Thus, $f_{1}\le f_{2}\le\cdots$ and $\lim_{k\to\infty}f_k(M)=\Vert M\Vert_{p}^{-\alpha}=\pi_{p,\alpha}(M)$ for every $M$. Therefore, by Theorem 3.4.8 of \cite{helms2009potential}, we need only prove that $f_k$ is superharmonic. Define $\epsilon=1/k$ and $f=\sum_{i=1}^m(\sigma_i^2+\epsilon)^{p/2}$. Then $f_k=f^\beta$.
Using Proposition \ref{matrix-lapla}, the Laplacian is
\begin{align}
  \Delta f_k  &=2\sum_{i<j}\frac{\sigma_i\frac{\partial f_k}{\partial \sigma_i}-\sigma_j\frac{\partial f_k}{\partial \sigma_j}}{\sigma_i^2-\sigma_j^2}+(n-m)\sum_{i=1}^m\frac{1}{\sigma_i}\frac{\partial f_k}{\partial \sigma_i}+\sum_{i=1}^m\frac{\partial^2 f_k}{\partial \sigma_i^2} \notag\\
  &=2\sum_{i<j}\frac{\beta p\sigma_i^2(\sigma_i^2+\epsilon)^{p/2-1} f^{\beta-1}-\beta p\sigma_j^2(\sigma_j^2+\epsilon)^{p/2-1} f^{\beta-1}}{\sigma_i^2-\sigma_j^2}+(n-m)\sum_{i=1}^m\beta p(\sigma_i^2+\epsilon)^{p/2-1} f^{\beta-1}\notag\\
  &\quad+\sum_{i=1}^m\beta pf^{\beta-2}\Big((\beta-1)p\sigma_i^2(\sigma_i^2+\epsilon)^{p-2} +f(\sigma_i^2+\epsilon)^{p/2-1}+(p-2)f\sigma_i^2(\sigma_i^2+\epsilon)^{p/2-2}\Big) \notag\\
  &=\beta pf^{\beta-2}\Big\{2f\sum_{i<j}\frac{\sigma_i^2(\sigma_i^2+\epsilon)^{p/2-1}-\sigma_j^2(\sigma_j^2+\epsilon)^{p/2-1}}{\sigma_i^2-\sigma_j^2}+(n-m+1)f\sum_{i=1}^m(\sigma_i^2+\epsilon)^{p/2-1} \notag\\
  &\quad+(\beta-1)p\sum_{i=1}^m \sigma_i^2(\sigma_i^2+\epsilon)^{p-2}+(p-2)f\sum_{i=1}^m\sigma_i^2(\sigma_i^2+\epsilon)^{p/2-2} \Big\}. \label{laplacian}
\end{align}

Using $\beta\le0$ and $0\le p \le 2$, the Laplacian \eqref{laplacian} is not greater than
\begin{equation}
\beta pf^{\beta-2}\Big\{2f\sum_{i<j}\frac{(\sigma_i^2+\epsilon)^{p/2}-(\sigma_j^2+\epsilon)^{p/2}}{\sigma_i^2-\sigma_j^2}+(n-m+p-1)f\sum_{i=1}^m(\sigma_i^2+\epsilon)^{p/2-1}+(\beta-1)p\sum_{i=1}^m(\sigma_i^2+\epsilon)^{p-1} \Big\}. \label{ineq-0}
\end{equation}

We first discuss the case $p\in[1,2]$. Using Lemma \ref{nns-inq} and treating $(\sigma_i^2+\epsilon)^{1/2}$ as $\sigma_i$, we have 
\begin{equation}
    f\sum_{i<j}\frac{(\sigma_i^2+\epsilon)^{p/2}-(\sigma_j^2+\epsilon)^{p/2}}{\sigma_i^2-\sigma_j^2}\ge\frac{pm(m-1)}{4}\sum_{i=1}^m(\sigma_i^2+\epsilon)^{p-1}. \label{ineq-1}
\end{equation}
Using Lemma~\ref{chebyshev} and $0\le p \le 2$, we have
\begin{equation}
    f\sum_{i=1}^m(\sigma_i^2+\epsilon)^{p/2-1}=\Big(\sum_{i=1}^m(\sigma_i^2+\epsilon)^{p/2}\Big)\Big(\sum_{i=1}^m(\sigma_i^2+\epsilon)^{p/2-1}\Big)\ge m\sum_{i=1}^m(\sigma_i^2+\epsilon)^{p-1}. \label{ineq-2}
\end{equation}
Because $\beta\le0$,
using \eqref{ineq-0}, \eqref{ineq-1}, and \eqref{ineq-2}, the Laplacian \eqref{laplacian} is not greater than
\begin{align}
    \beta pf^{\beta-2}\Big\{\frac{pm(m-1)}{2}\sum_{i=1}^m(\sigma_i^2+\epsilon)^{p-1}+(n-m+p-1)m\sum_{i=1}^m(\sigma_i^2+\epsilon)^{p-1}+(\beta-1)p\sum_{i=1}^m(\sigma_i^2+\epsilon)^{p-1} \Big\}. \label{lapla-2}
\end{align}
Because $\beta p=-\alpha\ge (1-p/2)m(m+1)+p-nm$, $$\frac{pm(m-1)}{2}+(n-m+p-1)m+(\beta-1)p\ge0.$$ Thus, \eqref{lapla-2} is nonpositive. Thus, the Laplacian \eqref{laplacian} is nonpositive.

Next, we discuss the case $p\in(0,1)$. Because $\alpha\in[0,nm-m(m+1)+(m-1)p]$, $n-m+p-1\ge0$. Using \eqref{ineq-0} and \eqref{ineq-2}, the Laplacian \eqref{laplacian} is not greater than
\begin{align}
    \beta pf^{\beta-2}\Big\{(n-m+p-1)m\sum_{i=1}^m(\sigma_i^2+\epsilon)^{p-1}+(\beta-1)p\sum_{i=1}^m(\sigma_i^2+\epsilon)^{p-1} \Big\}. \label{lapla-3}
\end{align}
Because $\beta p=-\alpha\ge -nm+m(m+1)-(m-1)p$, \eqref{lapla-3} is nonpositive. Thus, the Laplacian \eqref{laplacian} is nonpositive.

Because the Laplacian \eqref{laplacian} is nonpositive, using  Lemma 3.3.4 of \cite{helms2009potential}, $f_k$ is superharmonic, which completes the proof.
\end{proof}

\begin{remark}
    For case $p\in[1,2]$, We can show that $\Vert M\Vert_{p}^{-\alpha}$ is not superharmonic when $\alpha>nm-(1-p/2)m(m+1)-p$. Define $f=\sum_{i=1}^m\sigma_i^p$ and $\epsilon=0$. If $\sigma_i=N-i$ and $N$ is large enough, the Laplacian \eqref{laplacian} has the same sign as \eqref{lapla-2}. Because $\beta p=-\alpha<(1-p/2)m(m+1)+p-nm$, \eqref{lapla-2} is positive. Thus, the Laplacian \eqref{laplacian} can be positive in this case.
\end{remark}

\begin{remark}
    For case $p\in (0,1)$, there are instances where $\Vert M\Vert_{p}^{-\alpha}$ is also superharmonic for some $\alpha > nm-m(m+1)+(m-1)p$. For example, when $m=2$, $n=3$, $p=1/2$, and $\alpha=1$, $\Vert M\Vert_{p}^{-\alpha}$ is superharmonic. It can be proven by showing \eqref{ineq-0} is nonpositive. However, for general $m$, $n$, and $p$, it seems hard to obtain accurate upper bound of $\alpha$ that makes $\Vert M\Vert_{p}^{-\alpha}$ a superharmonic function.
\end{remark}

\bibliographystyle{acmtrans.bst}
\bibliography{scholar.bib}

\end{document}